\documentclass[10pt]{article}
\usepackage{amsthm, amstext}
\usepackage{amsmath}
\usepackage{amsmath, amsthm, amssymb}
\usepackage{pstricks}
\usepackage{enumitem}
\usepackage{graphicx}
\usepackage{amsfonts}
\usepackage{latexsym, xypic}
\usepackage[all]{xy}
\xyoption{all}
\usepackage[colorlinks,linkcolor=blue,citecolor=blue]{hyperref}
\theoremstyle{plain}
\newtheorem{theorem}{Theorem}[section]
\newtheorem{lemma}[theorem]{Lemma}
\newtheorem{proposition}[theorem]{Proposition}

\theoremstyle{definition}
\newtheorem{definition}[theorem]{Definition}

\newtheorem{example}{\sc Example}
\theoremstyle{remark}
\newtheorem{remark}{\sc Remark}
\theoremstyle{case}

\def\pi{positive implicative }

\date{}
\begin{document}

\title{\bf {Abnormal and Contranormal $L$-Subgroups of an $L$-group}}
\author{\textbf{Ananya Manas$^1$ and Iffat Jahan$^2$} \\\\ 
	$^{1}$Department of Mathematics, \\
	University of Delhi, Delhi, India \\
	ananyamanas@gmail.com \\  
	$^{2}$ Department of Mathematics, Ramjas College\\
	University of Delhi, Delhi, India \\
	ij.umar@yahoo.com \\\\
	  }
\date{}
\maketitle

\begin{abstract}
	\noindent In this paper, the concepts of abnormal and contranormal $L$-subgroups of an $L$-group have been introduced using the notion of the conjugate developed in \cite{jahan_conj}. Then, the properties of abnormal and contranormal $L$-subgroups have been studied analogous to their group-theoretic counterparts. Thereafter, several relations of abnormality and contranormality have been investigated in context of normality, maximality, normalizer and normal closure of $L$-subgroups of an $L$-group. \\\\
	{\bf Keywords:} $L$-algebra; $L$-subgroup; Generated $L$-subgroup; Normal $L$-subgroup; Maximal $L$-subgroup; Abnormal $L$-subgroup; Contranormal $L$-subgroup.
\end{abstract}

\section{Introduction}

\noindent In 1971, Rosenfeld \cite{rosenfeld_fuzzy} applied the notion of fuzzy sets, introduced by Zadeh \cite{zadeh_fuzzy}, to subgroups and subgroupoids, initiating the studies of fuzzy algebraic structures. Thereafter, Liu \cite{liu_op}, in 1981, extended this notion to lattices and introduced the notion of lattice valued fuzzy subgroups. Subsequently, a number of researchers investigated fuzzy algebraic structures and generalized various concepts to the fuzzy setting. We mention here that in majority of these studies, the parent structure was considered to be a classical group rather than an $L$(fuzzy)-group. This setting has a significant limitation that it does not allow for the formulation of various concepts in the $L$(fuzzy) theory. This drawback can be easily removed by taking the parent structure to be a $L$(fuzzy)-group rather than a classical group. Indeed, in \cite{ajmal_gen, ajmal_nc, ajmal_nil, ajmal_nor, ajmal_sol}, Ajmal and Jahan have introduced and studied various algebraic structures in $L$-setting specifically keeping in view their compatibility. Additionally, the authors of this paper have continued this research in \cite{jahan_max, jahan_app} by introducing the maximal and Frattini $L$-subgroups of an $L$-group and exploring their relationship with various structures in $L$-group theory. This paper is a continuation of similar studies. 

In classical group theory, both abnormal subgroups and contranormal subgroups are considered to be ``opposite" to the notion of normal subgroups \cite{carter_nilpotent, fattahi_groups, kurdachenko_abnormal}. This is because they satisfy two noteworthy properties that are in contrast to the properties of normal subgroups - firstly, the only subgroup of a group that is both normal and abnormal (contranormal) is the whole group, and secondly, maximal subgroups are either normal or abnormal (contranormal) subgroups. In the case of $L$(fuzzy)-group theory, no study has been undertaken to generalize the notion of abnormal or contranormal $L$-subgroups, as well as study their interactions with the well established notions of normal and maximal $L$-subgroups. This is because the notion of conjugate fuzzy subgroup, introduced by Mukherjee and Bhattacharya \cite{mukherjee_some}, was not adequate to define these concepts. However, in \cite{jahan_conj}, the authors have provided a new definition of conjugate $L$-subgroups, wherein the conjugate is taken with respect to an $L$-point rather than a crisp point of the parent group $G$. Moreover, this notion of the conjugate has been shown to be compatible with the notions of normality and normalizer of $L$-subgroups. In addition, this notion of the conjugate can be readily applied to define the abnormal and contranormal $L$-subngroups of an $L$-group. 

We begin our work in section 3 by defining the abnormal $L$-subgroup by an $L$-point using the notion of the conjugate developed in \cite{jahan_conj}. A level subset characterization for abnormal $L$-subgroups has been developed. Then, the image of abnormal $L$-subgroups under group homomorphism has been discussed. An example has been added to demonstrate the notion of abnormal $L$-subgroups. Thereafter, the notion of abnormal $L$-subgroups has been studied with respect to normality of $L$-subgroups. It has been shown that the only $L$-subgroup of an $L$-group $\mu$ that is both a normal and an abnormal $L$-subgroup of $\mu$ is $\mu$ itself. Then, it has been shown that an abnormal $L$-subgroup of $\mu$ is self-normalizing in $\mu$. Finally, the abnormality of maximal $L$-subgroups has been discussed.

In section 4, we define the contranormal $L$-subgroup of an $L$-group. An equivalent definition of contranormal $L$-subgroups is provided using the normal closure of $L$-subgroups of an $L$-group. Thereafter, a level subset characterization has been provided for contranormal $L$-subgroups. Then, various properties of contranormal $L$-subgroups have been investigated in context of normality and maximality of $L$-subgroups. Finally, we prove that an abnormal $L$-subgroup of $\mu$ is a contranormal $L$-subgroup of $\mu$. We finish the section by demonstrating this result with an example.     

\section{Preliminaries}

Throughout this paper, $L = \langle L, \leq, \vee, \wedge  \rangle$ denotes a completely distributive lattice where '$\leq$' denotes the partial ordering on $L$ and '$\vee$'and '$\wedge$' denote, respectively, the join (supremum) and meet (infimum) of the elements of $L$. Moreover, the maximal and minimal elements of $L$ will be denoted by $1$ and $0$, respectively. The concept of completely distributive lattice can be found in any standard text on the subject \cite{gratzer_lattices}. 

The notion of a fuzzy subset of a set was introduced by Zadeh \cite{zadeh_fuzzy} in 1965. In 1967, Goguen \cite{goguen_sets} extended this concept to $L$-fuzzy sets. In this section, we recall the basic definitions and results associated with $L$-subsets that shall be used throughout the rest of this work. These definitions can be found in chapter 1 of \cite{mordeson_comm}.

Let $X$ be a non-empty set. An $L$-subset of $X$ is a function from $X$  into $L$. The set of  $L$-subsets of $X$ is called  the $L$-power set of $X$ and is denoted by $L^X$.  For  $\mu \in L^X, $  the set $ \lbrace\mu(x) \mid x \in X \rbrace$  is called the image of $\mu$  and is denoted by  $\text{Im}(\mu)$. The tip and tail of $ \mu $  are defined as $\bigvee \limits_{x \in X}\mu(x)$ and $\bigwedge \limits_{x \in X}\mu(x)$, respectively. An $L$-subset $\mu$ of $X$ is said to be contained in an $L$-subset $\eta$  of $X$ if  $\mu(x)\leq \eta (x)$ for all $x \in X$. This is denoted by $\mu \subseteq \eta $.  For a family $\lbrace\mu_{i} \mid i \in I \rbrace$  of $L$-subsets in  $X$, where $I$  is a non-empty index set, the union $\bigcup\limits_{i \in I} \mu_{i} $    and the intersection  $\bigcap\limits_{i \in I} \mu_{i} $ of  $\lbrace\mu_{i} \mid i \in I \rbrace$ are, respectively, defined by
\begin{center}
	$\bigcup\limits_{i \in I} \mu_{i}(x)= \bigvee\limits_{i \in I} \mu(x)  $ \quad and \quad $\bigcap\limits_{i \in I} \mu_{i} (x)= \bigwedge\limits_{i \in I} \mu(x) $
\end{center}
for each  $x \in X $. If  $\mu \in L^X $  and  $a \in L $,  then the level  subset $\mu_{a}$ of $\mu$ is defined as
\[	\mu_{a}= \lbrace x \in X \mid \mu (x) \geq a\rbrace. \]
For $\mu, \nu \in L^{X} $, it can be verified easily that if $\mu\subseteq \nu$, then $\mu_{a} \subseteq \nu_{a} $ for each $a\in L $.

For $a\in L$ and $x \in X$, we define $a_{x} \in L^{X} $ as follows: for all $y \in X$,
\[
a_{x} ( y ) =
\begin{cases}
	a &\text{if} \ y = x,\\
	0 &\text{if} \ y\ne x.
\end{cases}
\]
$a_{x} $ is referred to as an $L$-point or $L$-singleton. We say that $a_{x} $ is an $L$-point of $\mu$ if and only if
$\mu( x )\ge a$ and we write $a_{x} \in \mu$. 

Let $S$ be a groupoid. The set product $\mu \circ \eta$   of $\mu, \eta \in L^S$ is an $L$-subset of $S$ defined by
\begin{center}
	$\mu \circ \eta (x) = \bigvee \limits_{x=yz}\lbrace\mu (y) \wedge \eta (z) \rbrace.$
\end{center}

\noindent Note that if $x$ cannot be factored as  $x=yz$  in $S$, then  $\mu \circ \eta (x)$, being  the least upper bound of the empty set, is zero. It can be verified that the set product is associative in  $L^S$  if $S$ is a semigroup.

Let $f$ be a mapping from a set $X$ to a set $Y$. If $\mu \in L ^{X}$ and $\nu \in L^{Y}$, then the image $f(\mu )$
of $\mu $ under $f$ and the preimage $f^{-1} (\nu )$ of $\nu $ under $f$ are $L$-subsets of $Y$ and $X$ respectively, defined by
\[
f(\mu )(y)=\bigvee\limits_{x\in f^{-1} (y)} \{\mu (x)\} \qquad\text{ and } \qquad f^{-1} (\nu )(x)=\nu (f(x)).
\]
Again,  if $f^{-1} (y)=\phi $,
then $f(\mu )(y)$ being the least upper bound of the empty set, is zero.

\begin{proposition}(\cite{mordeson_comm}, Theorem 1.1.14)
	\label{hom}
	Let $f : X \rightarrow Y$ be a mapping.
	\begin{enumerate}
		\item[({i})] Let $\{ \mu_i \}_{i \in I}$ be a family of $L$-subsets of $X$. Then, $f(\mathop{\cup}\limits_{i \in I} \mu_i) = \mathop{\cup}\limits_{i \in I} f(\mu_i)$ and $f(\mathop{\cap}\limits_{i \in I} \mu_i) \subseteq \mathop{\cap}\limits_{i \in I}f(\mu_i)$.
		\item[({ii})] Let $\mu \in L^X$. Then, $f^{-1}(f(\mu)) \supseteq \mu$. The equality holds if $f$ is injective.
		\item[{(iii)}] Let  $\nu \in L^Y$. Then, $f(f^{-1}(\nu)) \subseteq \nu$. The equality holds if $f$ is surjective.
		\item[{(iv)}] Let $\mu \in L^X$ and $\nu \in L^Y$. Then, $f(\mu) \subseteq \nu$ if and only if $\mu \subseteq f^{-1}(\nu)$. Moreover, if $f$ is injective, then $f^{-1}(\nu) \subseteq \mu$ if and only if $\nu \subseteq f(\mu)$.
	\end{enumerate}
\end{proposition}

Throughout this paper, $G$ denotes an ordinary group with the identity element `$e$' and $I$ denotes a non-empty indexing set. Also, $1_A$ denotes the characteristic function of a non-empty set $A$.

In 1971, Rosenfeld \cite{rosenfeld_fuzzy} applied the notion of fuzzy sets to groups to introduce the fuzzy subgroup of a group. Liu \cite{liu_op}, in 1983, extended the notion of fuzzy subgroups to $L$-fuzzy subgroups ($L$-subgroups), which we define below.   

\begin{definition}
	Let $\mu \in L ^G $. Then, $\mu $ is called an $L$-subgroup of $G$ if for each $x, y\in G$,
	\begin{enumerate}
		\item[({i})] $\mu (xy)\ge \mu (x)\wedge \mu (y)$,
		\item[({ii})] $\mu (x^{-1} )=\mu (x)$.
	\end{enumerate}
	The set of $L$-subgroups of $G$ is denoted by $L(G)$. Clearly, the tip of an $L$-subgroup is attained at the identity element of $G$, that is, $\text{tip}(\mu) = \mu(e)$.
\end{definition}

\begin{theorem}(\cite{mordeson_comm}, Lemma 1.2.5)
	\label{lev_gp}
	Let $\mu \in L ^G $. Then, $\mu $ is an $L$-subgroup of $G$ if and only if each non-empty level subset $\mu_{a} $ is a subgroup of $G$.
\end{theorem}

\begin{theorem}(\cite{mordeson_comm}, Theorems 1.2.10, 1.2.11)
	\label{hom_gp}
	Let $f : G \rightarrow H$ be a group homomorphism. Let $\mu \in L(G)$ and $\nu \in L(H)$. Then, $f(\mu) \in L(H)$ and $f^{-1}(\nu) \in L(G)$.
\end{theorem}

Let $\eta, \mu\in L^{{G}}$ such that $\eta\subseteq\mu$. Then, $\eta$ is said to be an $L$-subset of $\mu$. The set of all $L$-subsets of $\mu$ is denoted by $L^{\mu}.$
Moreover, if $\eta,\mu\in L(G)$ such that  $\eta\subseteq \mu$, then $\eta$ is said to be an $L$-subgroup of $\mu$. The set of all $L$-subgroups of $\mu$ is denoted by $L(\mu)$.

From now onwards, $\mu$ denotes an $L$-subgroup of $G$ which shall be considered as the parent $L$-group. 

\begin{definition}(\cite{ajmal_sol}) 
	Let $\eta\in L(\mu)$ such that $\eta$ is non-constant and $\eta\ne\mu$. Then, $\eta$ is said to be a proper $L$-subgroup of $\mu$.
\end{definition}

\noindent Clearly, $\eta$ is a proper $L$-subgroup of $\mu$ if and only if $\eta$ has distinct tip and tail and $\eta\ne\mu$.

\begin{definition}(\cite{ajmal_nil})
	Let $\eta \in L(\mu)$. Let $a_0$ and $t_0$ denote the tip and tail of $\eta$, respectively. We define the trivial $L$-subgroup of $\eta$ as follows:
	\[ \eta_{t_0}^{a_0}(x) = \begin{cases}
		a_0 & \text{if } x=e,\\
		t_0 & \text{if } x \neq e.
	\end{cases} \]
\end{definition}

\begin{theorem}(\cite{ajmal_nil}, Theorem 2.1)
	\label{lev_sgp}
	Let $\eta \in L^\mu$. Then, $\eta\in L(\mu)$ if and only if each non-empty level subset $\eta_a$  is a subgroup of $\mu_a$.	
\end{theorem}

\noindent It is well known in the literature that the intersection of an arbitrary family of $L$-subgroups of an $L$-group $\mu$ is an $L$-subgroup of $\mu$.
  
\begin{definition}(\cite{mordeson_comm})
	Let $\eta \in L^\mu $. Then, the $L$-subgroup of $\mu$ generated by $\eta$ is defined as the smallest $L$-subgroup of $\mu$ which contains $\eta $. It is denoted by $\langle \eta \rangle $, that is,
	\[ \langle \eta \rangle = \cap\{ \theta \in L(\mu) \mid \eta \subseteq \theta \}. \]
\end{definition}

\noindent The normal fuzzy subgroup of a fuzzy group was introduced by Wu \cite{wu_normal} in 1981. We note that for the development of this concept, Wu \cite{wu_normal} preferred  $L$-setting. Below, we recall the notion of a normal $L$-subgroup of an $L$-group:

\begin{definition}(\cite{wu_normal})
	Let $\eta \in L(\mu)$. Then, we say that  $\eta$  is a normal $L$-subgroup of $\mu$   if  
	\begin{center}
		$\eta(yxy^{-1}) \geq \eta(x)\wedge \mu(y)$ for  all  $x,y \in G.$
	\end{center}
\end{definition}

\noindent The set of normal $L$-subgroups of $\mu$  is denoted by $NL(\mu)$. If $\eta \in NL(\mu)$, then we write\vspace{.2cm} $ \eta \triangleleft \mu$. 

\noindent Here, we mention that the arbitrary intersection of a family of normal $L$-subgroups of an $L$-group $\mu$ is again a normal $L$-subgroup of $\mu$. 

\begin{theorem}(\cite{mordeson_comm}, Theorem 1.4.3)
	\label{lev_norsgp}
	Let $\eta \in L(\mu)$. Then, $\eta\in NL(\mu)$ if and only if each non-empty level subset $\eta_a$  is a normal subgroup of $\mu_a$.
\end{theorem}

\begin{definition}(\cite{mordeson_comm})
	Let $\mu \in L^X$. Then, $\mu$ is said to possess sup-propery if for each $A \subseteq X$, there exists $a_0 \in A$ such that $ \mathop \vee \limits_{a \in A}  {\mu(a) } = \mu(a_0)$. 
\end{definition}

\noindent Lastly, recall the following form \cite{ajmal_gen, ajmal_sol}:

\begin{theorem}(\cite{ajmal_gen}, Theorem 3.1)
	\label{gen}
	Let $\eta\in L^{^{\mu}}.$ Let $a_{0}=\mathop {\vee}\limits_{x\in G}{\left\{\eta\left(x\right)\right\}}$ and define an $L$-subset $\hat{\eta}$ of $G$ by
	\begin{center}
		$\hat{\eta}\left(x\right)=\mathop{\vee}\limits_{a \leq a_{0}}{\left\{a \mid x\in\left\langle \eta_{a}\right\rangle\right\}}$.
	\end{center}
	
	\noindent Then, $\hat{\eta}\in L(\mu)$ and  $\hat{\eta} =\left\langle \eta \right\rangle$.
\end{theorem}

\begin{theorem}(\cite{ajmal_gen}, Theorem 3.7)
	\label{gen_sup}
	Let $\eta \in L^{\mu}$ and $a_0 = \mathop{\vee}\limits_{x \in G}\{\eta(x)\}$, then for all $b \leq a_0$, $\langle \eta_b \rangle \subseteq \langle \eta \rangle_b$. Moreover, if $\eta$ possesses the sup-property, then $\langle \eta_b \rangle = \langle \eta \rangle_b$.
\end{theorem}

\begin{theorem}(\cite{ajmal_sol}, Lemma 3.27)
	\label{gen_hom}
	Let $f : G \rightarrow H$ be a group homomorphism, let $\mu \in L(G)$ and $\nu \in L(H)$. Then, for all $\eta \in L^{\mu}$, $\langle f(\eta) \rangle = f(\langle \eta \rangle)$ and for all $\theta \in L^{\nu}$, $\langle f^{-1}(\theta) \rangle = f^{-1}(\langle \theta \rangle).$
\end{theorem}
	
\section{Abnormal $L$-subgroups}

In classical group theory, abnormal subgroups are defined using the notion of the conjugate of a subgroup. In fuzzy group theory, the concept of the conjugate was introduced by Mukherjee and Bhattacharya \cite{mukherjee_some} in 1986. However, the conjugate discussed therein is by a crisp point of $G$ rather than a fuzzy point and thus cannot be utilized to develop significant notions in $L$-group theory. In \cite{jahan_conj}, the authors have introduced the concept of the conjugate $L$-subgroup by an $L$-point. This definition was shown to be highly compatible with the notions of normality and normalizer of $L$-subgroups of an $L$-group. Moreover, this new notion of the conjugate can be readily utilized to define abnormal $L$-subgroups of an $L$-group. We recall this defintion below:
	
\begin{definition}\label{dfncon}(\cite{jahan_conj})
	 Let $\eta$ be an $L$-subgroup of $\mu$ and $a_z \in \mu$. The conjugate of $\eta$ by $a_z$ is defined to be the $L$-subset $\eta^{a_z}$ given by
	\[ \eta^{a_z}(x) = a \wedge \eta(zxz^{-1}) \quad \text{ for all } x \in G. \]	
\end{definition}

\begin{remark}\label{conj_tip}
	For any $L$-subgroup $\eta$ of $\mu$ and $L$-point $a_z\in \mu$, $\eta^{a_z}$ is an $L$-subgroup of $\mu$. Moreover, $\text{tip}(\eta^{a_z}$) = $a \wedge \text{tip}(\eta)$, since
	\[ \eta^{a_z}(e) = a \wedge \eta(zez^{-1}) = a \wedge \eta(e). \]
\end{remark}

\noindent Below, we recall a level subset characterization for conjugate $L$-subgroups from \cite{jahan_conj}.

\begin{theorem}(\cite{jahan_conj})\label{lvl_conj}
	Let $\eta, \nu \in L(\mu)$ and $a \in L$ such that $\text{tip}(\nu) = a \wedge \text{tip}(\eta)$. Then, $\nu=\eta^{a_z}$ for $a_z \in \mu$ if and only if $\nu_t = {\eta_t}^{z^{-1}}$ for all $t \leq \text{tip}(\nu)$.
\end{theorem}

\noindent We are now ready to define abnormal $L$-subgroups of an $L$-group. Here, we note that for $L$-subsets $\eta, \nu$ of $\mu$, $\langle \eta, \nu \rangle$ is the $L$-subgroup of $\mu$ generated by $\eta \cup \nu$.

\begin{definition}
	Let $\eta$ be an $L$-subgroup of an $L$-group $\mu$. Then, $\eta$ is said to be an abnormal $L$-subgroup of $\mu$ if for every $L$-point $a_x \in \mu$, $a_x \in \langle \eta, \eta^{a_x} \rangle$. 
\end{definition}

\noindent In Theorem \ref{levsub1}, we provide a level subset characterization for abnormal $L$-subgroups. Firstly, we recall that a lattice $L$ is said to be upper well ordered if every non-empty subset of $L$ contains its supremum. 

\begin{theorem}
	\label{levsub1}
	Let $\eta$ be an $L$-subgroup of $\mu$ such that $\text{tip}(\eta) = \text{tip}(\mu)$. If $\eta_t$ is an abnormal subgroup of $\mu_t$ for all $t \leq \text{tip}(\eta)$, then $\eta$ is an abnormal $L$-subgroup of $\mu$. Conversely, let $L$ be an upper well ordered lattice and $\mu \in L(G)$. If $\eta$ is an abnormal $L$-subgroup of $\mu$, then $\eta_t$ is an abnormal subgroup of $\mu_t$ for all $t \leq \text{tip}(\eta)$.
\end{theorem}
\begin{proof}
	$(\Rightarrow)$ Suppose that $\text{tip}(\eta) = \text{tip}(\mu)$ and $\eta_t$ is an abnormal subgroup of $\mu_t$ for all $t \leq \text{tip}(\eta)$. To show that $\eta$ is an abnormal $L$-subgroup of $\mu$, let $a_x \in \mu$. Then, 
	\[ a \leq \mu(x) \leq \text{tip}(\mu) = \text{tip}(\eta). \]
	By hypothesis, $\eta_a$ is an abnormal $L$-subgroup of $\mu_a$. Moreover, $x^{-1} \in \mu_a$. Thus $x^{-1} \in \langle \eta_a, {\eta_a}^{x^{-1}} \rangle$, that is,
	\[ x^{-1} = x_1 x_2 \ldots x_k, \text{ where } x_i \text{ or } {x_i}^{-1} \in \eta_a \cup {\eta_a}^{x^{-1}}. \]
	Now, if $x_i \in \eta_a$, then $\eta(x_i) \geq a$. On the other hand, if $x_i \in {\eta_a}^{x^{-1}}$, then $xx_ix^{-1} \in \eta_a$, that is, $\eta(xx_ix^{-1}) \geq a$. Thus $\eta^{a_x}(x_i) \geq a$. Hence
	\[ x^{-1} = x_1 x_2 \ldots x_k, \text{ where } (\eta \cup \eta^{a_x})(x_i) \geq a, \] that is,
	\[ x^{-1} = x_1 x_2 \ldots x_k, \text{ where } x_i \text{ or } {x_i}^{-1} \in (\eta \cup \eta^{a_x})_a. \]
	This implies $x^{-1} \in \langle (\eta \cup \eta^{a_x})_a \rangle$.	Hence by Theorem \ref{gen}, 
	\[ \langle \eta, \eta^{a_x} \rangle (x^{-1}) = \mathop{\vee}\limits_{c \leq \eta(e)} {\left\{c \mid x^{-1} \in \langle (\eta \cup \eta^{a_x} )_c \rangle \right\}} \geq a. \]
	Since $\langle \eta, \eta^{a_x} \rangle$ is an $L$-subgroup of $\mu$, $\langle\eta, \eta^{a_x} \rangle (x) \geq a$. Thus $a_x \in \langle \eta, \eta^{a_x} \rangle$. We conclude that $\eta$ is an abnormal $L$-subgroup of $\mu$.  
	
	$(\Leftarrow)$ Conversely, let $L$ be an upper well ordered lattice and let $\mu \in L(G)$. Let $\eta$ be an abnormal $L$-subgroup of $\mu$ and let $t \leq \text{tip}(\eta)$. To show that $\eta_t$ is an abnormal subgroup of $\mu_t$, let $x \in \mu_t$. Then, $t_{x^{-1}} \in \mu$. By abnormality of $\eta$ in $\mu$, $t_{x^{-1}} \in \langle \eta, \eta^{t_{x^{-1}}} \rangle$, that is, $\langle \eta, \eta^{t_{x^{-1}}} \rangle (x^{-1}) \geq t.$ By Theorem \ref{gen},
	\begin{center} 
		$\langle \eta, \eta^{t_{x^{-1}}} \rangle(x^{-1}) = \mathop{\vee}\limits_{c \leq \eta(e)}{\left\{c \mid x^{-1} \in \left\langle (\eta \cup \eta^{t_{x^{-1}}})_c \right\rangle \right\}}.$ 
	\end{center}
	Let $A = \{ c \leq \eta(e) \mid x^{-1} \in\left\langle (\eta \cup \eta^{t_{x^{-1}}})_c \right\rangle \}$. Then, $A$ is a non-empty subset of $L$ since $0 \in L$. Moreover, since $L$ is upper well ordered, $A$ contains its supremum, say $c_0$. Thus $x^{-1} \in \langle (\eta \cup \eta^{t_{x^{-1}}})_{c_0} \rangle$ and $c_0 \geq t$. This implies $(\eta \cup \eta^{t_{x^{-1}}})_{c_0} \subseteq (\eta \cup \eta^{t_{x^{-1}}})_t$ and hence $x^{-1} \in \langle (\eta \cup \eta^{t_{x^{-1}}})_t \rangle$. Thus
	\[ x^{-1} = x_1 x_2 \ldots x_k, \text{ where } {x_i} \text{ or } {x_i}^{-1} \in (\eta \cup \eta^{t_{x^{-1}}})_t, \]
	that is, $(\eta \cup \eta^{t_{x^{-1}}})(x_i) \geq t$. This implies 
	\[ \eta(x_i) \vee \eta^{t_{x^{-1}}}(x_i) \geq t. \]
	Again, since $L$ is upper well ordered, $\eta(x_i) \geq t$ or $\eta^{t_{x^{-1}}}(x_i) \geq t$. If $\eta(x_i) \geq t$, then $x_i \in \eta_t$. On the other hand, if $\eta^{t_{x^{-1}}}(x_i) \geq t$, then 
	\[ \eta(x^{-1} x_i x) \geq t \wedge \eta(x^{-1} x_i x) \geq t, \]
	that is, $x_i \in {\eta_t}^x$. Thus $x_i \in \eta_t \cup {\eta_t}^x$. Therefore
	\[ x^{-1} = x_1 x_2 \ldots x_k, \text{ where } x_i \text{ or } {x_i}^{-1} \in \eta_t \cup {\eta_t}^x. \]
	This implies $x^{-1} \in \langle \eta_t, {\eta_t}^x \rangle$. Since $\langle \eta_t, {\eta_t}^x \rangle$ is a subgroup of $G$, we conclude that $x \in \langle \eta_t, {\eta_t}^x \rangle$.
\end{proof}

\noindent In Theorem \ref{hom_abn}, we discuss the image of an abnormal $L$-subgroup under a surjective group homomorphism. For this, we recall Lemma \ref{hom_conj} from \cite{jahan_conj}:

\begin{lemma}\label{hom_conj} (\cite{jahan_conj})
	Let $f : G \rightarrow H$ be a group homomorphism and $\mu \in L(G)$. Then, for $\eta \in L(\mu)$ and $a_z \in \mu$, the $L$-subgroup $f(\eta^{a_z})$ is a conjugate $L$-subgroup of $f(\eta)$ in $f(\mu)$. In fact, 
	\[ f(\eta^{a_z}) = f(\eta)^{a_{f(z)}}. \]
\end{lemma}

\begin{lemma}\label{hom_lev}
	Let $f:G \rightarrow H$ be a group homomorphism and $\mu \in L(G)$. Then, for $\eta \in L(\mu)$, 
	\[ f(\eta_t) \subseteq f(\eta)_t \]
	for all $t \leq \text{tip}(\eta)$.
\end{lemma}

\begin{theorem}\label{hom_abn}
	Let $f : G \rightarrow H$ be a surjective group homomorphism. Let $\mu \in L(G)$ such that $\mu$ possesses sup-property. If $\eta$ is an abnormal $L$-subgroup of $\mu$, then $f(\eta)$ is an abnormal $L$-subgroup of f($\mu$).
\end{theorem}
\begin{proof}
	Let $a_x \in f(\mu)$. We have to show that $a_x \in \langle f(\eta), f(\eta)^{a_x} \rangle$. Since $a_x \in f(\mu)$, $f(\mu)(x) \geq a$. By definition,
	\[ f(\mu)(x) = \vee \{ \mu(g) \mid g \in f^{-1}(x) \}. \]
	Let $A = \{ g \in G \mid g \in f^{-1}(x) \}$. Since $f$ is a surjection, $A$ is a non-empty subset of $G$. Since $\mu$ possesses the sup-property, there exists $s \in A$ such that 
	\[ a \leq f(\mu)(x) = \vee \{ \mu(g) \mid g \in A \} = \mu(s). \] 
	Hence $f(s) = x$ and $a_s \in \mu$. Now, since $\eta$ is an abnormal $L$-subgroup of $\mu$, $a_s \in \langle \eta, \eta^{a_s} \rangle$, that is, $\langle \eta, \eta^{a_s} \rangle(s) \geq a$. By Theorem \ref{gen},
	\[  \langle \eta, \eta^{a_s} \rangle(s) = \mathop{\vee}\limits_{c \leq \eta(e)}\left\{c \mid s \in \langle(\eta \cup \eta^{a_s})_c \rangle \right\} \]
	Let $c \leq \eta(e)$ such that $s \in \langle (\eta \cup \eta^{a_s})_c \rangle$. Then,
	\[ s = s_1 s_2 \ldots s_n, \text{ where } s_i \text{ or } {s_i}^{-1} \in (\eta \cup \eta^{a_s})_c. \]
	This implies
	\[ x = f(s) = f(s_1)f(s_2)\ldots f(s_n), \text{ where } f(s_i) \text{ or } f(s_i)^{-1} \in f((\eta \cup \eta^{a_s})_c).\]
	By Lemma \ref{hom_lev}, $f((\eta \cup \eta^{a_s})_c) \subseteq (f(\eta \cup \eta^{a_s}))_c = (f(\eta) \cup f(\eta^{a_s}))_c$. Also, by Theorem \ref{hom_conj}, $(f(\eta^{a_s})) = f(\eta)^{a_{f(s)}} = f(\eta)^{a_x}$. Hence
	\[ x = f(s_1)f(s_2)\ldots f(s_n), \text{ where } f(s_i) \text{ or } f(s_i)^{-1} \in (f(\eta) \cup f(\eta)^{a_x})_c, \]
	that is, $x \in \langle f(\eta) \cup f(\eta^{a_x}))_c \rangle$. Thus
	\begin{equation*}
		\begin{split}
			\langle f(\eta), f(\eta)^{a_x} \rangle (x) &= \mathop{\vee}\limits_{c \leq f(\eta)(e)}\left\{c \mid x \in \langle(f(\eta) \cup f(\eta)^{a_x})_c \rangle \right\}\\
			&\geq \mathop{\vee}\limits_{c \leq \eta(e)}\left\{c \mid s \in \langle(\eta \cup \eta^{a_s})_c \rangle \right\}\\
			&= \langle \eta, \eta^{a_s} \rangle (s)\\
			&\geq a.  
		\end{split}
	\end{equation*}
	Hence $a_x \in \langle f(\eta), f(\eta)^{a_x} \rangle$. We conclude that $f(\eta)$ is an abnormal $L$-subgroup of $f(\mu)$.
\end{proof}

\noindent As a motivation for the definition of abnormal $L$-subgroups, we have the following result:

\begin{proposition}
	Let $H$ and $K$ be subgroups of $G$ such that $H \subseteq K$. Then, $H$ is an abnormal subgroup of $K$ if and only if $1_H$ is an abnormal $L$-subgroup of $1_K$.
\end{proposition}

\noindent Below, we demonstrate abnormal $L$-subgroups of an $L$-group with an example:

\begin{example}
	Let $G = S_4$ and $L$ be lattice given by Figure 1. 
	\begin{center}
		\includegraphics[scale=0.5]{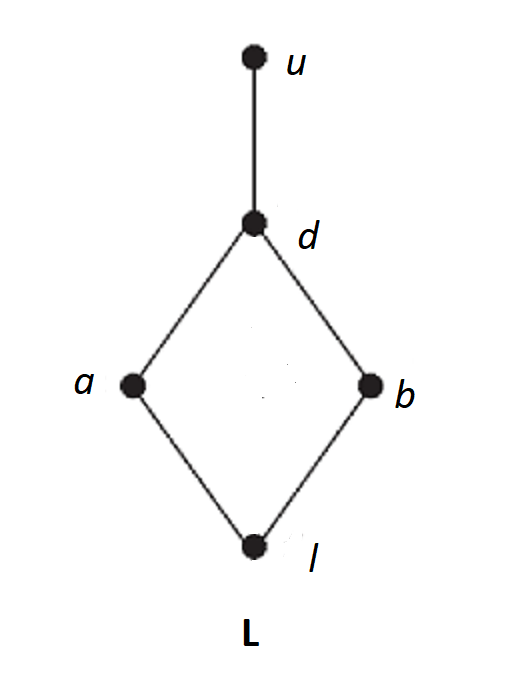}\\
		\centering \text{Figure 1}		
	\end{center}
	Let 
	\[H_1 = \{ \epsilon, (12), (13), (23), (123), (132) \}, H_2 = \{ \epsilon, (12), (14), (24), (124), (142) \}, \]
	\[ H_3 = \{ \epsilon, (14), (13), (14), (134), (143) \} \text{ and } H_4 = \{ \epsilon, (23), (24), (34), (234), (243) \} \]
	be the four copies of $S_3$ in $S_4$, where $\epsilon$ denotes the identity permutation.
	Define an $L$-subgroup $\mu$ of $S_4$ as follows:
	\[ \mu(x) = \begin{cases}
		u &\text{if } x \in \langle (12) \rangle; \\
		d &\text{otherwise}. 
	\end{cases} \]
	Since $\mu_t$ is a subgroup of $G$ for all $t \leq \mu(e)$, it follows from Theorem \ref{lev_gp} that $\mu \in L(G)$. Now, define an $L$-subset $\eta$ of $\mu$ by
	\[ \eta(x) = \begin{cases}
		u &\text{if } x \in \langle (12) \rangle; \\
		a &\text{if } x \in H_1 \setminus \langle (12) \rangle; \\
		b &\text{if } x \in H_2 \setminus \langle (12) \rangle; \\
		l &\text{otherwise}.
	\end{cases} \]
	Again, since each $\eta_t$ is a subgroup of $\mu_t$, by Theorem \ref{lev_sgp}, $\eta \in L(\mu)$. We show that $\eta$ is an abnormal $L$-subgroup of $\mu$. 
	
	Firstly, consider the $L$-point $d_{(34)} \in \mu$. The conjugate $L$-subgroup $\eta^{d_{(34)}}$ is given by
	\begin{equation*}
	\begin{split}
		\eta^{d_{(34)}} (x) &= d \wedge \eta((34)x(34)) \\
		&= \begin{cases}
			d &\text{if } x \in \langle (12) \rangle; \\
			a &\text{if } x \in H_2 \setminus \langle (12) \rangle; \\
			b &\text{if } x \in H_1 \setminus \langle (12) \rangle; \\
			l &\text{otherwise.}
		\end{cases}
	\end{split}
	\end{equation*}
	Now, by Theorem \ref{gen},
	\[ \langle \eta, \eta^{d_{(34)}} \rangle = \mathop{\vee}\limits_{t \leq u} \{ t \mid x \in \langle (\eta \cup \eta^{d_{(34)}})_t \rangle  \}. \]
	Note that $\eta_a = H_1, \eta_b = H_2, (\eta^{d_{(34)}})_a = H_2$ and $(\eta^{d_{(34)}})_b = H_1$. Moreover,
	\[ (34) = (24)(132)(124) \in \langle H_1, H_2 \rangle. \]
	Hence
	\[ \langle \eta, \eta^{d_{(34)}} \rangle (34) = a \vee b = d. \]
	Thus $d_{(34)} \in \langle \eta, \eta^{d_{(34)}} \rangle$.  
	Similarly, it can be easily verified that $c_x \in \langle \eta, \eta^{c_x} \rangle$ for all $c_x \in \mu$. Hence $\eta$ is an abnormal $L$-subgroup of $\mu$.  
\end{example}

\noindent We now explore the various interactions of abnormal $L$-subgroups with the notions of normality and normalizer of $L$-subgroups of an $L$-group. More importantly, we generalize the various properties of abnormal subgroups to the $L$-setting. We begin by showing that the only $L$-subgroup af an $L$-group that is both normal and abnormal is the whole $L$-group (Theorem \ref{abn_nor}). Firstly, we recall the following from \cite{jahan_conj}: 

\begin{lemma}(\cite{jahan_conj})
	\label{nor_conj}
	Let $\eta \in L(\mu)$. Then, $\eta$ is a normal $L$-subgroup of $\mu$ if and only if  $\eta^{a_z} \subseteq \eta$ for every $L$-point $a_z \in \mu$. Moreover, if $\eta \in NL(\mu)$ and $\text{tip}(\eta^{a_z}) = \text{tip}(\eta)$, then $\eta^{a_z} = \eta$. 
\end{lemma}

\begin{theorem}
	\label{abn_nor}
	Let $\mu \in L(G)$. Then, the only $L$-subgroup of $\mu$ that is both normal and abnormal in $\mu$ is $\mu$ itself.
\end{theorem}
\begin{proof}
	Let $\eta \in L(\mu)$ such that $\eta$ is both a normal as well as an abnormal $L$-subgroup of $\mu$. Let $x \in G$ and $a=\mu(x)$. Then, $a_x \in \mu$. By the abnormality of $\eta$, $a_x \in \langle \eta, \eta^{a_x} \rangle$. Now, since $\eta$ is a normal $L$-subgroup of $\mu$, it follows from Lemma \ref{nor_conj} that $\eta^{a_x} \subseteq \eta$. Thus
	\[ a_x \in \langle \eta, \eta^{a_x} \rangle = \eta. \]
	Therefore $\mu(x) = a \leq \eta(x)$, that is, $\eta(x) = \mu(x)$. We conclude that $\eta = \mu$.
\end{proof}

\begin{proposition}\label{inc_lvl}
	Let $\eta, \nu \in L(\mu)$. Then, $\eta \subseteq \nu$ if and only if $\eta_t \subseteq \nu_t$ for all $t \leq \text{tip}(\eta)$.
\end{proposition}

\begin{proposition}(\cite{jahan_conj})\label{conj_inv}
	Let $\eta \in L(\mu)$ and $a_x$ be an $L$-point of $\mu$. Then, $\eta^{a_x} \subseteq \eta$ if and only if $\eta^{a_{x^{-1}}} \subseteq \eta$.
\end{proposition}
\begin{proof}
	Suppose that $\eta^{a_z} \subseteq \eta$. Then, $(\eta^{a_z})_t \subseteq \eta_t$ for all $t \leq \text{tip }(\eta^{a_z}) = a \wedge \eta(e)$. By Theorem \ref{lvl_conj}, $(\eta^{a_z})_t = {\eta_t}^{z^{-1}}$ for all $t \leq a \wedge \eta(e)$. Thus
	\[ {\eta_t}^{z^{-1}} \subseteq \eta_t \text{ for all } t \leq a \wedge \eta(e). \]
	This implies $z^{-1} \in N(\eta_t)$ for all $t \leq a \wedge \eta(e)$. Since the normalizer of $\eta_t$ is a subgroup of $G$, $z \in N(\eta_t)$ for all $t \leq a \wedge \eta(e)$. Thus
	\[ {\eta_t}^z \subseteq \eta_t \text{ for all } t \leq a \wedge \eta(e). \]
	Again, by Theorem \ref{lvl_conj}, $(\eta^{a_{z^{-1}}})_t = {\eta_t}^z$ for all $t \leq a \wedge \eta(e) $. Thus $(\eta^{a_{z^{-1}}})_t \subseteq \eta_t$ for all $t \leq a \wedge \eta(e)= \text{tip }(\eta^{a_z})$. Hence by Proposition \ref{inc_lvl}, $\eta^{a_{z^{-1}}} \subseteq \eta$.
	
	\noindent  The converse part can be obtained by replacing $z$ by $z^{-1}$ in the above exposition.	
\end{proof}

\noindent The notion of the normalizer of an $L$-subgroup was introduced by Ajmal and Jahan in \cite{ajmal_nor} using the concept of cosets of $L$-subgroups. In \cite{jahan_conj}, the authors have given a new definition of the normalizer using the conjugate of $L$-subgroups, like their classical counterparts. For the sake of completenes, we have provide a different and exhaustive proof of the new definition in Theorem \ref{def_norm}.

\begin{theorem}\label{def_norm}
	Let $\eta \in L(\mu)$ and let $\lambda$ be the $L$-subset of $G$ defined as follows:
	\[ \lambda = \bigcup \left\{ a_z \in \mu \mid \eta^{a_z} \subseteq \eta \right\}. \]
	Then, $\lambda$ is the largest $L$-subgroup of $\mu$ such that $\eta$ is a normal $L$-subgroup of $\lambda$. Hence $\lambda$ is the normalizer of $\eta$ in $\mu$.
\end{theorem}
\begin{proof}
	Firstly, note that for all $x \in G$, 
	\[ \lambda(x) = \bigcup \left\{ a_z \in \mu \mid \eta^{a_z} \subseteq \eta \right\}(x) = \bigvee_{\substack{a_x \in \mu \\ \eta^{a_x} \subseteq \eta}} \left\{ a \right\}. \]
	For $x \in G$, define the subset $L_{\eta}(x)$ of $L$ as follows:
	\[ L_{\eta}(x) = \{ a \in L \mid a_x \in \mu \text{ and } \eta^{a_x} \subseteq \eta \}. \]
	Then, $\lambda(x) = \vee \{ a \mid a \in L_{\eta}(x) \}$.
	Now, we show that $\lambda$ is an $L$-subgroup of $\mu$. Clearly, $\lambda \subseteq \mu$, since for all $x \in G$, 
	\[ \lambda(x) = \vee \{ a \mid a_x \in \mu \text{ and } \eta^{a_x} \subseteq \eta \} \leq \vee \{ a \mid a_x \in \mu\} = \mu(x). \]  
	Let $x , y \in G$. We claim that if $a \in L_{\eta}(x)$ and $b \in L_{\eta}(y)$, then $a \wedge b \in L_{\eta}(xy)$. So, let $a \in L_{\eta}(x)$ and $b \in L_{\eta}(y)$. Then, $a_x, b_y \in \mu$ and $\eta^{a_x}, \eta^{b_y} \subseteq \eta$. Then, $(a \wedge b)_{xy} \in \mu$, since
	\[ \mu(xy) \geq \mu(x) \wedge \mu(y) \geq a \wedge b. \]
	Moreover, for all $g \in G$,
	\begin{equation*}
	\begin{split}
		\eta^{(a \wedge b)_{xy}}(g) &= (a \wedge b) \wedge \eta((xy)g(xy)^{-1}) \\
		&= b \wedge (a \wedge \eta(x(ygy^{-1})x^{-1})) \\
		&= b \wedge \eta^{a_x}(ygy^{-1}) \\
		&\leq b \wedge \eta(ygy^{-1}) \qquad\qquad\qquad\qquad\qquad (\text{since $\eta^{a_x} \subseteq \eta$}) \\
		&= \eta^{b_y}(g) \\
		&\leq \eta(g). ~\qquad\qquad\quad\qquad\qquad\qquad\qquad (\text{since $\eta^{b_y} \subseteq \eta$})
	\end{split}
	\end{equation*} 
	Hence $\eta^{(a \wedge b)_{xy}} \subseteq \eta$. Therefore $a \wedge b \in L_{\eta}(xy)$. Thus we have
	\begin{equation*}
	\begin{split}
		\lambda(xy) &= \vee \{ c \mid c \in L_{\eta}(xy) \} \\
		&\geq \{ a \wedge b \mid a \in L_{\eta}(x) \text{ and } b \in L_{\eta}(y) \} \\
		&= (\vee \{a \mid a \in L_{\eta}\}) \wedge (\{ b \mid b \in L_{\eta}(y)\}) \\
		&= \lambda(x) \wedge \lambda(y),
	\end{split}
	\end{equation*}
	since $L$ is a completely distributive lattice. Moreover, it follows readily by Proposition \ref{conj_inv} that $\lambda(x^{-1}) = \lambda(x)$. Thus $\lambda$ is an $L$-subgroup of $\mu$.
	
	\noindent Next, to show that $\eta$ is a $L$-subgroup of $\lambda$, let $x \in G$. Then, $(\eta(x))_x \in \mu$ and for all $g \in G$,
	\begin{equation*}
	\begin{split}
		\eta(g) &= \eta(x^{-1}(xgx^{-1})x) \\
		&\geq \eta(x^{-1}) \wedge \eta(xgx^{-1}) \\
		&= \eta(x) \wedge \eta(xgx^{-1}) \\
		&= \eta^{(\eta(x))_x}(g). 
	\end{split}
	\end{equation*}
	Thus $\eta^{(\eta(x))_x} \subseteq \eta$. Hence $\eta(x) \in L_{\eta}(x)$. Therefore
	\[ \eta(x) \leq \vee \{ a \mid a \in L_{\eta}(x) \} = \lambda(x). \] 
	Now, to show that $\eta$ is a normal $L$-subgroup of $\lambda$, let $x, y \in G$ and let $a \in L_{\eta}(y)$. Then, $\eta^{a_y} \subseteq \eta$. By Proposition \ref{conj_inv}, $\eta^{a_{y^{-1}}} \subseteq \eta$. Thus 
	\begin{equation*}
	\begin{split}
		a \wedge \eta(x) &= a \wedge \eta(y^{-1}(yxy^{-1})y) \\
		&= \eta^{a_{y^{-1}}}(yxy^{-1}) \\
		&\leq \eta(yxy^{-1}).
	\end{split}
	\end{equation*}
	Therefore
	\begin{equation*}
	\begin{split}
		\eta(yxy^{-1}) &\geq \vee \{ \eta(x) \wedge a \mid L_{\eta}(y) \} \\
		&= \eta(x) \wedge \{ \vee \{ a \mid a \in L_{\eta}(y) \}\} \\ 
		&= \eta(x) \wedge \lambda(y).
	\end{split}
	\end{equation*}
	Thus $\eta$ is a normal $L$-subgroup of $\mu$.
	\noindent Finally, we show that $\lambda$ is the largest $L$-subgroup of $\mu$ such that $\eta$ is normal in $\lambda$. To this end, let $\theta$ be an $L$-subgroup of $\mu$ containing $\eta$ such that $\eta$ is a normal $L$-subgroup of $\theta$. Let $x \in G$. Then, $(\theta(x))_x \in \mu$ and for all $g \in G$,
	\begin{equation*}
	\begin{split}
		\eta^{(\theta(x))_x}(g) &= \theta(x) \wedge \eta(xgx^{-1}) \\
		&= \theta(x^{-1}) \wedge \eta(xgx^{-1}) ~~~\qquad\qquad\qquad (\text{since $\theta \in L(\mu)$} )\\
		&\leq \eta(x^{-1}(xgx^{-1})x) \qquad\qquad\qquad\qquad (\text{since $\eta \in NL(\theta)$})\\
		&= \eta(g).
	\end{split}
	\end{equation*}   
	Hence $(\theta(x))_x \in L_{\eta}(x)$. Therefore
	\[ \theta(x) \leq \vee \{ a \mid a \in L_{\eta}(x) \} = \lambda(x). \]
	We conclude that $\theta \subseteq \lambda$.
\end{proof}

\noindent As a corollary to Theorem \ref{def_norm}, we note that an $L$-subgroup $\eta$ of $\mu$ is a normal $L$-subgroup of $\mu$ if and only if $N(\eta) = \mu$. In Theorem \ref{norm_abn}, we show that an abnormal $L$-subgroup of $\mu$ is a self-normalizing $L$-subgroup of $\mu$. These two facts together provide another distinction between the notions of abnormal and normal $L$-subgroups.   

\begin{theorem}
	\label{norm_abn}
	Let $\eta$ be an abnormal $L$-subgroup of $\mu$. Then, $\eta$ is a self-normalizing subgroup of $\mu$, that is, $N(\eta) = \eta$.
\end{theorem}
\begin{proof}
	Let $\eta$ be an abnormal $L$-subgroup of $\mu$. To show that $N(\eta) \subseteq \eta$, let $a_z \in \mu$ such that $\eta^{a_z} \subseteq \eta$. Then, 
	\[ a_z \in \langle \eta, \eta^{a_z} \rangle = \eta. \]
	It follows from Theorem \ref{def_norm} that
	\[ N(\eta) = \bigcup \left\{ a_z \in \mu \mid \eta^{a_z} \subseteq \eta \right\} \subseteq \eta. \]
	Hence $N(\eta) = \eta$.
\end{proof}

\begin{remark}
	In \cite{jahan_nil}, the authors have proved that nilpotent $L$-subgroups of an $L$-group satisfy the normalizer condition, that is, if $\eta$ is a nilpotent $L$-subgroup of an $L$-group $\mu$, then $\eta$ is a proper $L$-subgroup of its normalizer $N(\eta)$. Hence nilpotent $L$-subgroups cannot be self-normalizing subgroups. It follows from the above theorem that a nilpotent $L$-subgroup cannot be an abnormal $L$-subgroup. Hence the class of nilpotent $L$-subgroups is disjoint from the class of abnormal $L$-subgroups.
\end{remark}

\noindent Here, we provide an example to demonstrate the above theorem.

\begin{example}
	Consider the $L$-subgroup $\eta$ of the $L$-group $\mu$ discussed in Example 1. We determine the normalizer $N(\eta)$ of $\eta$. By Theorem \ref{def_norm}, 
	\[ N(\eta) = \bigcup \left\{ c_z \in \mu \mid \eta^{c_z} \subseteq \eta \right\}, \]
	that is, for $x \in G$,
	\[ N(\eta)(x) = \bigvee_{\substack{c_x \in \mu \\ \eta^{c_x} \subseteq \eta}} \left\{ c \right\}. \]
	Using the above expression, we evaluate $N(\eta)$ for some particular elements. Firstly, consider $(23) \in G$. Let $c_{(23)} \in \mu$. Then, 
	\begin{equation*}
	\begin{split}
		\eta^{c_{(23)}} (x) &= c \wedge \eta((23)x(23)) \\ 
		&= \begin{cases}
			c \wedge u &\text{if } x \in \langle (13) \rangle; \\
			c \wedge a &\text{if } x \in H_1 \setminus \langle (13) \rangle; \\
			c \wedge b &\text{if } x \in H_3 \setminus \langle (13) \rangle;\\
			c \wedge l &\text{otherwise}
		\end{cases}
	\end{split}
	\end{equation*}
	Thus $\eta^{c_{(23)}} \subseteq \eta$ only if $c = a$ or $c = l$. Therefore $N(\eta)(23) = a = \eta(a)$. Similarly, it can be easily seen that $N(\eta)(x) = \eta(x)$ for all $x \in G$. Hence $N(\eta) = \eta$. 
\end{example}

\noindent Finally, we show that every maximal $L$-subgroup of an $L$-group $\mu$ is either a normal or an abnormal $L$-subgroup of $\mu$.

Below, we recall the definition of maximal $L$-subgroups from \cite{jahan_max}:

\begin{definition}(\cite{jahan_max})
	Let $\mu \in L(G)$. A proper $L$-subgroup $\eta$ of $\mu$ is said to be a maximal $L$-subgroup of $\mu$ if, whenever $\eta \subseteq \theta \subseteq \mu$ for some $\theta \in L(\mu)$, then either $\theta = \eta$ or $\theta = \mu$.
\end{definition}

\noindent Below, we recall the notion of jointly supstar $L$-subsets from \cite{ajmal_nil}. It is worthwile to mention here that this notion is a generalization of the notion of sup-property and lends itself easily for applications.

\begin{definition}
	A non-empty subset $X$ of a lattice $L$ is said to be supstar if every non-empty subset $A$ of $X$ contains its supremum, that is, $\mathop{\vee}\limits_{a \in A} \{a\} \in A$.
\end{definition}

\noindent The following result follows immediately from the above definition:

\begin{proposition}
	Let $\eta\in L^{\mu}$. Then, $\eta$ possesses sup-property if and only if $\text{Im}(\eta)$ is a supstar subset of $L$.
\end{proposition}

\begin{definition}
	Let $\{\eta_i:i\in I\}\subseteq L^{\mu}$ be an arbitrary family. Then, $\{\eta_i:i\in I\}$ is said to be a supstar family if $\mathop{\bigcup}\limits_{i\in I}\text{Im} (\eta_i)$ is a supstar subset of $L$. In particular, two $L$-subsetes $\eta, \theta \in L^{\mu}$ are said to be jointly supstar if $\text{Im}(\eta) \cup \text{Im}(\theta)$ is a supstar subset of $L$.
\end{definition}

\begin{lemma}
	\label{conj_sup}
	Let $\eta \in L(\mu)$ such that $\eta$ and $\mu$ are jointly supstar. Let $a_x \in \mu$, where $a \in \text{Im}(\mu)$. Then, 
	\[ \text{Im}(\eta^{a_x}) \subseteq \text{Im}(\eta)\cup\text{Im}(\mu). \]
	Moreover, $\eta \cup \eta^{a_x}$ possesses sup-property.
\end{lemma}

\begin{theorem}
	Let $\eta \in L(\mu)$ such that $\eta$ and $\mu$ are jointly supstar. Let $\eta$ be a maximal $L$-subgroup of $\mu$. Then, $\eta$ is either a normal or an abnormal $L$-subgroup of $\mu$.
\end{theorem}
\begin{proof}
	Firstly, since $\eta$ is a maximal $L$-subgroup of $\mu$, it is a proper $L$-subgroup of $\mu$. Hence by Theorem \ref{abn_nor}, $\eta$ cannot be both a normal and an abnormal $L$-subgroup of $\mu$. Now, suppose on the contrary, that $\eta$ is neither a normal nor an abnormal $L$-subgroup of $\mu$. Then, by Lemma \ref{nor_conj}, there exists an $L$-point $a_x \in \mu$ such that $\eta^{a_x} \nsubseteq \eta$. Also, it follows from the definition of abnormality that there exists an $L$-point $b_y \in \mu$ such that $b_y \notin \langle \eta, \eta^{b_y} \rangle$. Here, we note that $a_x \notin \eta$, for if $\eta(x) \geq a$, then for all $g \in G$,
	\begin{equation*}
		\begin{split}
			\eta(g) &= \eta(x^{-1}(xgx^{-1})x) \\
			&\geq\eta(x) \wedge \eta(xgx^{-1}) \qquad\qquad\qquad~~ (\text{since } \eta \in L(\mu))  \\
			&\geq a \wedge \eta(xgx^{-1}) \qquad\qquad\qquad\qquad (\text{since } \eta(x) \geq a) \\
			&= \eta^{a_x}(g),
		\end{split}
	\end{equation*}
	which contradicts our assumption that $\eta^{a_x} \nsubseteq \eta$. Also, it is clear that $b_y \notin \eta$. Let $c = \mu(x)$ and $d = \mu(y)$. Then, since $a_x \notin \eta$, $c_x \notin \eta$. Similarly, $d_y \notin \eta$. Moreover, since $\{ c, d \} \subseteq \text{Im}(\mu)$,
	\[ c \vee d = c \text{ or } d. \]
	Thus we have the following cases:
	\begin{description}
		\item[Case 1: $\mathbf{c<d}$.]	Define an $L$-subset $\theta$ of $G$ as follows:
		\[ \theta(g) = \{ \eta(g) \vee c \} \wedge \mu(g) \qquad\qquad \text{for all } g \in G. \]
		Note that for all $g \in G$, $\theta(g) \leq \mu(g)$. Hence $\theta \in L^{\mu}$. Now, we show that $\theta \in L(\mu)$. To this end, let $g_1, g_2 \in G$. Then,
		\begin{equation*}
			\begin{split}
				\theta(g_1g_2) &= \{ \eta(g_1g_2) \vee c \} \wedge \mu(g_1g_2) \\
				&\geq \{ \{ \eta(g_1) \wedge \eta(g_2) \} \vee c \} \wedge \{ \mu(g_1) \wedge \mu(g_2) \} \\
				&= \{\{\eta(g_1) \vee c\} \wedge \{\eta(g_2) \vee c\}\} \wedge \{\mu(g_1) \wedge \mu(g_2)\} \\
				&= \{\{ \eta(g_1) \vee c \} \wedge \mu(g_1)\} \wedge \{{\eta(g_2) \vee c } \wedge \mu(g_2) \} \\
				&=  \theta(g_1) \wedge \theta(g_2).
			\end{split}
		\end{equation*}
		Also, for all $g \in G$,
		\begin{equation*}
			\begin{split}
				\theta(g^{-1}) &= \{ \eta(g^{-1}) \vee c \} \wedge \mu(g^{-1}) \\
				&= \{ \eta(g) \vee c \} \wedge \mu(g) \qquad\qquad\qquad (\text{since } \eta, \mu \in L(G)) \\
				&= \theta(g).
			\end{split}
		\end{equation*}
		Hence $\theta \in L(\mu)$. Moreover, note that 
		\[ \eta(g) \leq \eta(g) \vee c \qquad \text{and} \qquad \eta(g) \leq \mu(g) \]
		for all $g \in G$. Hence
		\[ \eta(g) \leq \{ \eta(g) \vee c \} \wedge \mu(g) = \theta(g) \leq \mu(g) \]
		for all $g \in G$. Hence $\eta \subseteq \theta \subseteq \mu$. Now, since $\eta$  and $\mu$ are jointly supstar and $\{ \eta(x), c \} \subseteq \text{Im}(\mu) \cup \text{Im}(\eta)$, $\eta(x) \vee c = \eta(x) \text{ or } c$. Moreover, $c_x \notin \eta$ implies $\eta(x) < c$. Thus $\eta(x) \vee c = c$. This implies
		\[ \eta(x) < c = \{ \eta(x) \vee c \} \wedge \mu(x) = \theta(x). \]
		Hence $\eta \subsetneq \theta$. Next, since $\eta$  and $\mu$ are jointly supstar and $\{ \eta(y), d \} \subseteq \text{Im}(\mu) \cup \text{Im}(\eta)$, $\eta(y) \vee d = \eta(y) \text{ or } d$. Moreover, $d_y \notin \eta$ implies $\eta(y) < d$. Also, by assuption, $c < d$. Thus $\eta(y) \vee c < d$. This implies
		\[ \theta(y) = \{ \eta(y) \vee c \} \wedge \mu(y) < d = \mu(y). \]
		Thus $\theta \subsetneq \mu$. Hence we have found an $L$-subgroup $\theta$ of $\mu$ such that $\eta \subsetneq \theta \subsetneq \mu$, which contradicts the maximality of $\eta$ in $\mu$. 
		
		\item[Case 2: $\mathbf{d < c}$.] Define an $L$-subset $\theta$ of $G$ as follows:
		\[ \theta(g) = \{ \eta(g) \vee d \} \wedge \mu(g) \qquad\qquad \text{for all } g \in G. \]
		Then, it can be proven similar to case 1 that $\theta \in L(\mu)$ and $\eta \subseteq \theta \subseteq \mu$.
		Now, since $\eta$  and $\mu$ are jointly supstar and $\{ \eta(y), d \} \subseteq \text{Im}(\mu) \cup \text{Im}(\eta)$, $\eta(y) \vee d = \eta(y) \text{ or } d$. Moreover, $d_y \notin \eta$ implies $\eta(y) < d$. Thus $\eta(y) \vee d = d$. This implies
		\[ \eta(y) < d = \{ \eta(y) \vee d \} \wedge \mu(y) = \theta(y). \]
		Hence $\eta \subsetneq \theta$. Next, since $\eta$  and $\mu$ are jointly supstar and $\{ \eta(x), c \} \subseteq \text{Im}(\mu) \cup \text{Im}(\eta)$, $\eta(x) \vee c = \eta(x) \text{ or } c$. Moreover, $c_x \notin \eta$ implies $\eta(x) < c$. Also, by assuption, $d < c$. Thus $\eta(x) \vee d < c$. This implies
		\[ \theta(x) = \{ \eta(x) \vee d \} \wedge \mu(x) < c = \mu(x). \]
		Thus $\theta \subsetneq \mu$. Hence we have found an $L$-subgroup $\theta$ of $\mu$ such that $\eta \subsetneq \theta \subsetneq \mu$, which contradicts the maximality of $\eta$ in $\mu$. 
		
		\item[Case 3: $\mathbf{c=d}$ and $\mathbf{x \in \langle \eta_c, {\eta_c}^{y^{-1}} \rangle}$.] Consider the $L$-subgroup $\langle \eta, \eta^{c_y} \rangle$ of $\mu$. Note that 
		\[ \eta \subseteq \langle \eta, \eta^{c_y} \rangle \subseteq \mu. \]
		Moreover, since $\eta$ and $\mu$ are jointly supstar and $c \in \text{Im}(\mu)$, by Lemma \ref{conj_sup}, $\eta \cup \eta^{c_y}$ possesses the sup-property. Hence by Theorems \ref{gen_sup} and \ref{lvl_conj}, 
		\[ \langle \eta, \eta^{c_y} \rangle_c = \langle (\eta \cup \eta^{c_y} )_c \rangle = \langle \eta_c, {\eta_c}^{y^{-1}} \rangle. \]
		Hence $x \in  \langle \eta, \eta^{c_y} \rangle_c$, that is $c_x \in  \langle \eta, \eta^{c_y} \rangle$. Since $c_x \notin \eta$,
		\[ \eta \subsetneq \langle \eta, \eta^{c_y} \rangle \subseteq \mu. \]
		Therefore by maximality of $\eta$ in $\mu$, $\langle \eta, \eta^{c_y} \rangle = \mu$. Thus $c_y \in \langle \eta, \eta^{c_y} \rangle$. This implies 
		\[ y \in  \langle \eta, \eta^{c_y} \rangle_c = \langle \eta_c, {\eta_c}^{y^{-1}} \rangle. \]
		Now, $b \leq c$ implies $\eta_c \subseteq \eta_b$ and ${\eta_c}^{y^{-1}} \subseteq {\eta_b}^{y^{-1}}$. Thus
		\[  y \in \langle \eta, \eta^{c_y} \rangle_c = \langle (\eta \cup \eta^{c_y} )_c \rangle = \langle \eta_c, {\eta_c}^{y^{-1}} \rangle \subseteq \langle \eta_b, {\eta_b}^{y^{-1}} \rangle \subseteq \langle (\eta \cup \eta^{b_y})_b \rangle \subseteq \langle \eta, \eta^{b_y} \rangle_b, \]
		which contradicts the assumption that $b_y \notin \langle \eta, \eta^{b_y} \rangle$.
		
		\item[Case 4: $\mathbf{c=d}$ and $\mathbf{x \notin \langle \eta_c, {\eta_c}^{y^{-1}} \rangle}$.] Consider the $L$-subgroup $\langle \eta, \eta^{c_y} \rangle$ of $\mu$. Again, since $\eta$ and $\mu$ are jointly supstar and $c \in \text{Im}(\mu)$, by Lemma \ref{conj_sup}, $\eta \cup \eta^{c_y}$ possesses the sup-property. Hence by Theorems \ref{gen_sup} and \ref{lvl_conj}, 
		\[ \langle \eta, \eta^{c_y} \rangle_c = \langle (\eta \cup \eta^{c_y} )_c \rangle = \langle \eta_c, {\eta_c}^{y^{-1}} \rangle. \]
		Hence $x \notin  \langle \eta, \eta^{c_y} \rangle_c$, that is $c_x \notin  \langle \eta, \eta^{c_y} \rangle$. Hence
		\[ \eta \subseteq \langle \eta, \eta^{c_y} \rangle \subsetneq \mu. \]
		By maximality of $\eta$ in $\mu$, we must have $\eta = \langle \eta, \eta^{c_y} \rangle$. Thus $\eta^{c_y} \subseteq \eta$. Now, let $N(\eta)$ be the normalizer of $\eta$ in $\mu$. Then, $\eta \subseteq N(\eta) \subseteq \mu$. Moreover, since $\eta^{c_y} \subseteq \eta$, by Theorem \ref{def_norm}, $c_y \in N(\eta)$. Hence $\eta \subsetneq N(\eta)$. Moreover, $\eta^{a_x} \subsetneq \eta$ implies $a_x \notin N(\eta)$. Hence $N(\eta) \subsetneq \mu$. Thus
		\[ \eta \subsetneq N(\eta) \subsetneq \mu, \]
		which contradicts the maximality of $\eta$ in $\mu$.  
	\end{description}
	Thus we arrive at a contradiction in all of these exhaustive cases. Therefore we conclude that $\eta$ is either a normal or an abnormal $L$-subgroup of $\mu$. 
\end{proof}

\begin{theorem}
	Let $\eta$ be a maximal $L$-subgroup of $\mu$. Then, $\eta$ is either a normal or an abnormal $L$-subgroup of $\mu$.
\end{theorem}
\begin{proof}
	Let $\eta$ be a maximal $L$-subgroup of $\mu$ and let $N(\eta)$ denote the normalizer of $\eta$ in $\mu$. Then,
	\[ \eta \subseteq N(\eta) \subseteq \mu. \]
	By maximality of $\eta$, either $N(\eta) = \mu$ or $N(\eta) = \eta$. If $N(\eta) = \mu$, then $\eta$ is a normal $L$-subgroup of $\mu$. On the other hand, suppose that $N(\eta) = \eta$. We claim that $\eta$ is an abnormal $L$-subgroup of $\mu$. To this end, let $a_x \in \mu$. Then, we have the following cases:
	\begin{description}
		\item[Case 1: $\mathbf{\eta^{a_x} \nsubseteq \eta.}$]	Then, $\eta \subsetneq \langle \eta, \eta^{a_x} \rangle \subseteq \mu$. By maximality of $\eta$, $\langle \eta, \eta^{a_x} \rangle = \mu$. Thus $a_x \in \langle \eta, \eta^{a_x} \rangle$.
		\item[Case 2: $\mathbf{\eta^{a_x} \subseteq \eta}$.] Then, by Theorem \ref{def_norm}, $a_x \in N(\eta)$. Since $N(\eta) = \eta$, $a_x \in \langle \eta, \eta^{a_x} \rangle$.
	\end{description}
	\noindent Hence in both the cases, $a_x \in \langle \eta, \eta^{a_x} \rangle$. Thus $\eta$ is an abnormal $L$-subgroup of $\mu$.
\end{proof}

\noindent The following example illustrates the above theorem:

\begin{example}
	Let $M=\{ l,f,a,b,c,d,u \}$ be the lattice given by figure 1. Let $\mathbf{2}$ be the chain $0<1$. Then, 
	\begin{align*}
		M \times \mathbf{2} = &\{ (l,0), (f,0), (a,0), (b,0), (c,0), (d,0), (u,0),\\
		&~ (l,1), (f,1), (a,1), (b,1), (c,1), (d,1), (u,1) \}. 
	\end{align*} 
	\noindent Let $G=S_4$, the group of all permutations of the set $\{1,2,3,4\}$ with identity element $\epsilon$. Let
	\[ D_4^1 = \langle (24), (1234) \rangle, D_4^2 = \langle (12), (1324) \rangle, D_4^3 = \langle (23), (1342) \rangle \]
	denote the dihedral subgroups of $G$ and $V_4 = \{\epsilon, (12)(34), (13)(24), (14)(23)\}$ denote the Klein-4 subgroup of $G$. 
	\begin{center}
		\includegraphics[scale=1.0]{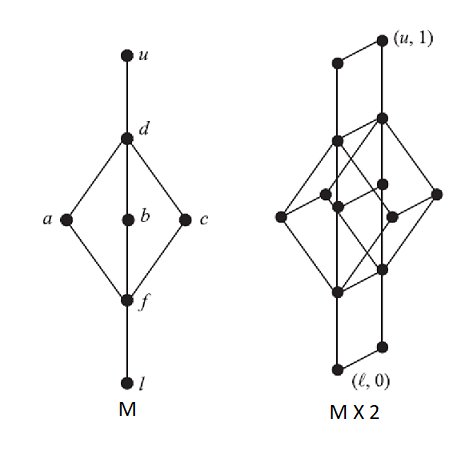}\\
		\centering \text{Figure 1}		
	\end{center}
	Define the $L$-subset $\mu$ of $G$ as follows:
	\[ \mu(x) = \begin{cases}
		(u,1) &\text{if } x = e,\\
		(d,1) &\text{if } x \in V_4 \setminus \{e\}, \\
		(d,0) &\text{if } x \in S_4 \setminus V_4.
	\end{cases} \]
	Since each non-empty level subset $\mu_t$ is a subgroup of $G$, by Theorem \ref{lev_gp}, $\mu$ is an $L$-subgroup of $G$.
	
	\noindent Now, define $\eta \in L^{\mu}$ as follows:
	\[ \eta(x) = \begin{cases}
		(u,1) &\text{if } x=e,\\
		(d,1) &\text{if } x \in V_4 \setminus \{e\},\\
		(a,0) &\text{if } x \in D_4^1 \setminus V_4,\\
		(b,0) &\text{if } x \in D_4^2 \setminus V_4,\\
		(c,0) &\text{if } x \in D_4^3 \setminus V_4,\\
		(f,0) &\text{if } x \in S_4 \setminus \mathop{\cup}\limits_{i=1}^3 D_4^i.
	\end{cases} \]
	Since every non-empty level subset $\eta_t$ is a subgroup of $\mu_t$, by Theorem \ref{lev_sgp}, $\eta \in L(\mu)$. Moreover, it can be verified that $\eta$ is a maximal $L$-subgroup of $\mu$. To this end, suppose that there exists an $L$-subgroup $\theta \in L(\mu)$ such that $\eta \subsetneq \theta \subseteq \mu$. Then, there exists $x \in G$ such that $\eta(x) < \theta(x) \leq \mu(x)$. Since 
	\[ \eta_{(u,1)} = \mu_{(u,1)} \qquad \text{and} \qquad \eta_{(d,1)} = \mu_{(d,1)}, \]
	we must have $x \in S_4 \setminus V_4$. Now, suppose that $x \in D_4^1 \setminus V_4$. Then, 
	$\eta(x) < \theta(x) \leq \mu(x)$ implies $\theta(x) = (d,0)$. But then, $D_4^1, D_4^2 \subseteq \theta_{(b,0)}$ implies $\theta_{(b,0)} = S_4$. Similarly, $\theta_{(c,0)} = S_4$. Thus $\theta_{(d,0)} = S_4$, which implies $\theta = \mu$. Moreover, it can be verified in a similar manner for all other $x$ that $\theta = \mu$. Hence $\eta$ is a maximal $L$-subgroup of $\mu$.
	
	\noindent Next, note that the pairs $(\eta_t, \mu_t)$ for $t = (a,1), (b,1) \text{ and } (c,1)$ are 
	\[ (\eta_{(a,1)}, \mu_{(a,1)}) = (D_4^1, S_4), (\eta_{(b,1)}, \mu_{(b,1)}) = (D_4^2, S_4), (\eta_{(c,1)}, \mu_{(c,1)}) = (D_4^3, S_4) \]
	Since none of the $D_4^i$ is a normal subgroup of $S_4$, by Theorem \ref{lev_norsgp}, $\eta$ is not a normal $L$-subgroup of $\mu$.  
		
	\noindent Finally, it can be shown easily that $\eta$ is an abnormal $L$-subgroup of $\mu$. For instance, consider the $L$-point ${(d,0)}_{(123)} \in \mu$. Then, by definition, 
	\[ \eta^{{(d,0)}_{(123)}}(x) = (d,0) \wedge \eta((123)x(132)) \]
	for all $x \in S_4$. Hence
	\[	\eta^{{(d,0)}_{(123)}}(x) = \begin{cases}
		(d,0) \wedge (u,1) &\text{if } x=e,\\
		(d,0) \wedge (d,1) &\text{if } x \in V_4 \setminus \{e\},\\
		(d,0) \wedge (b,0) &\text{if } x \in D_4^1 \setminus V_4,\\
		(d,0) \wedge (c,0) &\text{if } x \in D_4^2 \setminus V_4,\\
		(d,0) \wedge (a,0) &\text{if } x \in D_4^3 \setminus V_4,\\
		(d,0) \wedge (f,0) &\text{if } x \in S_4 \setminus \mathop{\cup}\limits_{i=1}^3 D_4^i.
	\end{cases} \]
	Thus
	\[\eta^{{(d,0)}_{(123)}}(x) = \begin{cases}
		(d,0) &\text{if } x \in V_4,\\
		(b,0) &\text{if } x \in D_4^1 \setminus V_4,\\
		(c,0) &\text{if } x \in D_4^2 \setminus V_4,\\
		(a,0) &\text{if } x \in D_4^3 \setminus V_4,\\
		(f,0) &\text{if } x \in S_4 \setminus \mathop{\cup}\limits_{i=1}^3 D_4^i.
	\end{cases} \] 
	Hence it is clear that ${(d,0)}_{(123)} \in \langle \eta, \eta^{{(d,0)}_{(123)}} \rangle$.
\end{example}

\section{Contranormal $L$-subgroups of an $L$-group}

\noindent In this section, we introduce the notion of a contranormal $L$-subgroup of an $L$-group and explore its various interactions with normality similar to abnormal $L$-subgroups. This provides a justification to the notion that contranormal $L$-subgroups serve as an opposite to normal $L$-subgroups.

Below, we define the contranormal $L$-subgroups of an $L$-group.

\begin{definition}
	Let $\eta$ be an $L$-subgroup of $\mu$. Then, $\eta$ is said to be a contranormal $L$-subgroup of $\mu$ if there does not exist any proper $L$-subgroup of $\mu$ containing every conjugate $\eta^{a_z}$ of $\eta$ in $\mu$.
\end{definition}

\noindent In Theorem \ref{con_sub}, we discuss the subnormality of contranormal $L$-subgroups of an $L$-group. This result illustrates the compability of these two notions in the context of $L$-group theory. Morover, using this result, we can formulate an alternate definition of contranormal $L$-subgroups, which we discuss in Definition \ref{def_contra}.

Here, we recall the definition of normal closure and subnormality from \cite{ajmal_nc}:.

\begin{definition}(\cite{ajmal_nc})
	Let $\eta \in L(\mu)$. The $L$-subset $\mu\eta\mu^{-1}$ of $\mu$ defined by 
	\[ \mu\eta\mu^{-1}(x) = \bigvee_{x=zyz^{-1}} \left\{ \eta(y) \wedge \mu(z) \right\} \qquad \text{ for each } x \in G \]
	is called the conjugate of $\eta$ in $\mu$.	The normal closure of $\eta$ in $\mu$, denoted by $\eta^\mu$, is defined to be the $L$-subgroup of $\mu$ generated by the conjugate $\mu\eta\mu^{-1}$, that is,
	\[ \eta^\mu = \langle \mu\eta\mu^{-1} \rangle. \]
	Moreover, $\eta^\mu$ is the smallest normal $L$-subgroup of $\mu$ containing $\eta$.
\end{definition}

\begin{remark}
	The tip of the normal closure $\eta^\mu$ of $\eta$ is same as $\text{tip}(\eta)$, that is, $\eta^\mu (e) = \eta(e)$. 
\end{remark}

\begin{definition}(\cite{ajmal_nc})
	Let $\eta \in L(\mu)$. Define a descending series of $L$-subgroups of $\mu$ inductively as follows:
	\[ \eta_0 = \mu \qquad \text{ and } \qquad \eta_i = \eta^{\eta_{i-1}} \qquad \text{ for all } i \geq 1. \]
	Then, $\eta_i$ is the smallest normal $L$-subgroup of $\eta_{i-1}$ containing $\eta$, called the $i^{th}$ normal closure of $\eta$ in $\mu$. The series of $L$-subgroups
	\[ \mu = \eta_0 \supseteq \eta_1 \supseteq \ldots \supseteq \eta_{i-1} \supseteq \eta_i \supseteq \ldots \]
	is called the normal closure series of $\eta$ in $\mu$. Moreover, if there exists a non-negative integer $m$ such that 
	\[ \eta = \eta_m \vartriangleleft \eta_{m-1} \vartriangleleft \ldots \vartriangleleft \eta_0 = \mu, \]
	then $\eta$ is said to be a subnormal $L$-subgroup of $\mu$ with defect $m$. 
	
	\noindent Clearly, $m=0$ if $\eta = \mu$ and $m=1$ if $\eta \in NL(\mu)$ and $\eta \neq \mu$.
\end{definition}

\begin{theorem}
	\label{con_sub}
	Let $\eta$ be a contranormal $L$-subgroup of $\mu$. Then, $\eta$ is a normal $L$-subgroup of $\mu$ if and only if $\eta = \mu$. Moreover, if $\eta$ is a proper $L$-subgroup of $\mu$, then $\eta$ is not a subnormal $L$-subgroup of $\mu$. Hence $\eta$ is either a subnormal $L$-subgroup of $\mu$ with defect $0$ or $\eta$ is not a subnormal $L$-subgroup of $\mu$.
\end{theorem}
\begin{proof}
	Firstly,  suppose that $\eta$ is a normal $L$-subgroup of $\mu$. Then, since $\eta$ is normal in $\mu$, by Lemma \ref{nor_conj}, $\eta^{a_z} \subseteq \eta$ for all $a_z \in \mu$. However, since $\eta$ is a contranormal $L$-subgroup of $\mu$, there does not exist any proper $L$-subgroup of $\mu$ containing every conjugate of $\eta$. It follows that $\eta = \mu$. 
	
	Next, suppose that $\eta$ is a proper $L$-subgroup of $\mu$. We claim that $\eta$ is a contranormal $L$-subgroup of $\mu$ if and only if $\eta^{\mu} = \mu$. For this, we show that if for $\theta \in L(\mu)$, $\eta^{a_z} \subseteq \theta$ for all $a_z \in \mu$ if and only if $\eta^\mu \subseteq \theta$.
 	
 	Suppose that  $\eta^{a_z} \subseteq \theta$ for all $a_z \in \mu$. Let $x \in G$. Then,
	\begin{equation*}
	\begin{split}
		\mu\eta\mu^{-1} (x) &= \bigvee_{\substack{x=zyz^{-1} \\ y,z \in G}} \left\{ \eta(y) \wedge \mu(z) \right\} \\
		&= \bigvee_{z \in G} \left\{ \eta(z^{-1}xz) \wedge \mu(z) \right\} \\
		&= \bigvee_{\substack{a \leq \mu(z) \\ z \in G}} \left\{ \eta(z^{-1}xz) \wedge a \right\} \\
		&= \left\{ \bigcup_{a_z \in \mu} \eta^{a_z} \right\} (x).
	\end{split}
	\end{equation*}
	By assumption, $\eta^{a_z} \subseteq \theta$ fo all $a_z \in \mu$. Hence $\mathop{\bigcup}\limits_{a_z \in \mu} \eta^{a_z} \subseteq \theta$. Hence $\eta^\mu = \langle \mu \eta \mu^{-1} \rangle \subseteq \theta$. 
	
	\noindent Conversely, suppose that $\eta^\mu \subseteq \eta$. Let $a_z \in \mu$. Then, $\mu(z^{-1}) \geq a$. Therefore for all $x \in G$,
	\begin{equation*}
	\begin{split}
		\eta^{a_z}(x) &= a \wedge \eta(zxz^{-1}) \\
		&\leq \mu(z^{-1}) \wedge \eta(zxz^{-1}) \\
		&\leq \bigvee_{\substack{x = gyg^{-1} \\ g,y \in G}} \left\{ \eta(y) \wedge \mu(g) \right\} \\
		&= \mu\eta\mu^{-1}(x).
	\end{split}
	\end{equation*}
	Since $\eta^\mu \subseteq \theta$, $\mu\eta\mu^{-1} \subseteq \theta$. This implies $\eta^{a_z} \subseteq \theta$ for all $a_z \in \mu$.
	
	From the above claim, it follows that $\eta$ is a contranormal $L$-subgroup of $\mu$ if and only if $\eta^\mu = \mu$. Now, let $\eta_i$ denote the $i^{th}$ normal closure of $\eta$ in $\mu$, where
	\[ \eta_0 = \mu \qquad \text{ and } \qquad \eta_i = \eta^{\eta_{i-1}} \qquad \text{ for all } i \geq 1. \]
	It follows that 
	\[ \eta_0 = \mu, \eta_1 = \eta^{\eta_0} = \eta^\mu, \eta_2 = \eta^{\eta_1} = \eta^{\mu} = \mu, \ldots, \eta_i = \eta^{\eta_{i-1}} = \eta^{\mu} = \mu, \ldots \]
	Hence $\eta_m \neq \eta$ for any $m \geq 0$. It follows that $\eta$ is not a subnormal $L$-subgroup of $\mu$.
	
\end{proof}

\noindent In view of the above theorem, we have the following equivalent definition of contranormal $L$-subgroups:

\begin{definition}
	\label{def_contra}
	Let $\eta \in L(\mu)$. Then, $\eta$ is said to be a contranormal $L$-subgroup of $\mu$ if the normal closure $\eta^\mu$ of $\eta$ in $\mu$ is the whole $L$-group $\mu$.
\end{definition}

\noindent In Theorem \ref{lev_contra}, we develop a level subset characterization for contranormal $L$-subgroups. 

\begin{lemma}
	\label{nc_sup}
	Let $\eta \in L(\mu)$ such that $\eta$ and $\mu$ are jointly supstar. Then, 
	\[ \text{Im}(\mu\eta\mu^{-1}) \subseteq \text{Im}(\eta)\cup\text{Im}(\mu). \]
	Moreover, if $L$ is a chain, then $\mu\eta\mu^{-1}$ possesses sup-property.
\end{lemma}

\begin{theorem}
	\label{lev_contra}
	Let $L$ be a chain and $\mu \in L(G)$. Let $\eta \in L(\mu)$ such that $\eta$ and $\mu$ are jointly supstar. If $\eta$ is a contranormal $L$-subgroup of $\mu$, then $\eta_t$ is a contranormal subgroup of $\mu_t$ for all $t \leq \text{tip}(\eta)$. Conversely, let $\mu \in L(G)$ and $\eta \in L(\mu)$ such that $\text{tip}(\eta) = \text{tip}(\mu)$. If $\eta_t$ is a contranormal $L$-subgroup of $\mu_t$ for all $t \leq \text{tip}(\eta)$, then $\eta$ is a contranormal $L$-subgroup of $\mu$. 
\end{theorem}
\begin{proof}
	($\Rightarrow$) Let $t \leq \text{tip}(\eta)$. To show that $\eta_t$ is a contranormal subgroup of $\mu_t$, we claim that $\{\eta^\mu\}_t = {\eta_t}^{\mu_t}$, where ${\eta_t}^{\mu_t}$ denotes the normal closure of $\eta_t$ in $\mu_t$.
	
	Firstly, note that $\eta^\mu$ is a normal $L$-subgroup of $\mu$ containing $\eta$. Hence by Theorem \ref{lev_norsgp}, $\{\eta^\mu\}_t$ is a normal subgroup of $\mu_t$ containing $\eta_t$. By the definition of normal closure, ${\eta_t}^{\mu_t}$ is the smallest normal subgroup of $\mu_t$ containing $\eta_t$. Hence ${\eta_t}^{\mu_t} \subseteq \{ \eta^\mu \}_t$. For the reverse inclusion, let $x \in \{ \mu\eta\mu^{-1} \}_t$. Then, $\mu\eta\mu^{-1}(x) \geq t$, that is,
	\[ \bigvee_{\substack{x = gyg^{-1} \\ g,y \in G}} \left\{ \eta(y) \wedge \mu(g) \right\} \geq t. \]
	Let $A = \{ \eta(y) \wedge \mu(g) \mid y,g \in G \text{ and } gyg^{-1} = x \}.$ Then, $A$ is a non-empty subset of $\text{Im}(\mu)\cup\text{Im}(\eta)$. Since $\eta$ and $\mu$ are jointly supstar, $A$ contains its supremum, that is, there exist $y_1, g_1 \in G$ such that
	\[ \bigvee_{\substack{x = gyg^{-1} \\ g,y \in G}} \left\{ \eta(y) \wedge \mu(g) \right\} = \eta(y_1) \wedge \mu(g_1). \]
	Moreover, $\eta(y_1) \wedge \mu(g_1) \geq t$, that is, $y_1 \in \eta_t$ and $g_1 \in \mu_t$. Thus 
	\[ x = g_1y_1{g_1}^{-1} \in {\eta_t}^{\mu_t}. \]
	Hence $\{\mu\eta\mu^{-1}\}_t \subseteq {\eta_t}^{\mu_t}$. This implies $\langle \{\mu\eta\mu^{-1}\}_t \rangle \subseteq {\eta_t}^{\mu_t} $. Now, by Lemma \ref{nc_sup}, $\mu\eta\mu^{-1}$ possesses sup-property. Hence by Theorem \ref{gen_sup}, 
	\[ \{\eta^\mu\}_t = \langle \mu\eta\mu^{-1} \rangle_t = \langle \{ \mu\eta\mu^{-1}\}_t \rangle \subseteq {\eta_t}^{\mu_t}. \]
	Thus $\{\eta^\mu\}_t = {\eta_t}^{\mu_t}$. This proves the claim.
	
	Now, since $\eta$ is a contranormal $L$-subgroup of $\mu$, $\eta^\mu = \mu$. Hence $\{\eta^\mu\}_t = \mu_t$. It follows from the above claim that ${\eta_t}^{\mu_t} = \mu_t$. Thus $\eta_t$ is a contranormal subgroup of $\mu_t$.
	
	($\Leftarrow$) To show that $\eta$ is a contranormal $L$-subgroup of $\mu$, suppose that there exists an $L$-subgroup $\theta$ of $\mu$ such that $\eta^{a_z} \subseteq \theta$ for all $a_z \in \mu$. Note that $\text{tip}(\eta) = \text{tip}(\theta) = \text{tip}(\mu)$. We show that $\theta_t = \mu_t$ for all $t \leq \theta(e)$.
	
	\noindent Let $t \leq \theta(e)$. Since $\theta \subseteq \mu$, $\theta_t \subseteq \mu_t$. For the reverse inclusion, we claim that every conjugate of $\eta_t$ in $\mu_t$ is contained in $\theta_t$. To this end, let $x \in \mu_t$. Then, $t_{x^{-1}} \in \mu$. By assumption, $\eta^{t_{x^{-1}}} \subseteq \theta$. Thus $(\eta^{t_{x^{-1}}})_t \subseteq \theta_t$. By Theorem \ref{lvl_conj}, 
	\[ (\eta^{t_{x^{-1}}})_t = {\eta_t}^x. \]
	Therefore ${\eta_t}^x \subseteq \theta_t$, which proves the claim. Since $\eta_t$ is a contranormal subgroup of $\mu_t$, we must have $\theta_t = \mu_t$. Thus $\theta = \mu$. 
\end{proof}

\noindent The following result is immediate. We state it here without proof.

\begin{theorem}
	Let $H$ and $K$ be subgroups of $G$. Then, $H$ is a contranormal subgroup of $K$ if and only if $1_H$ is contranormal $L$-subgroup of $1_K$.
\end{theorem}

\begin{theorem}
	Let $\eta$ be a maximal $L$-subgroup of $\mu$. Then, $\eta$ is either a normal or a contranormal $L$-subgroup of $\mu$.
\end{theorem}
\begin{proof}
	Let $\eta$ be a maximal $L$-subgroup of $\mu$ and let $\eta^\mu$ be the normal closure of $\eta$ in $\mu$. Then,
	\[ \eta \subseteq \eta^\mu \subseteq \mu. \]
	By maximality of $\eta$, either $\eta^\mu = \mu$ or $\eta^\mu = \eta$. If $\eta^\mu = \eta$, then $\eta$ is a normal $L$-subgroup of $\mu$. On the other hand, if $\eta^\mu = \mu$, then by Definition \ref{def_contra}, $\eta$ is a contranormal $L$-subgroup of $\mu$. Hence the result. 
\end{proof}

\noindent Below, we show that an abnormal $L$-subgroup of $\mu$ is a contranormal $L$-subgroup of $\mu$.

\begin{theorem}
	Let $\eta$ be an abnormal $L$-subgroup of $\mu$. Then, $\eta$ is a contranormal $L$-subgroup of $\mu$.
\end{theorem}
\begin{proof}
	Let $\eta$ be an abnormal $L$-subgroup of $\mu$. Then, for all $a_z \in \mu$, $a_z \in \langle \eta, \eta^{a_z} \rangle$. To show that $\eta$ is a contranormal $L$-subgroup of $\mu$, let $\theta \in L(\mu)$ such that $\eta^{a_z} \subseteq \theta$ for all $a_z \in \mu$. We show that $\theta = \mu$.
	
	\noindent Firstly, note that $\eta \subseteq \theta$, since $\eta = \eta^{{a_0}_e} \subseteq \theta$, where $a_0 = \text{tip}(\mu)$. Now, let $x \in G$ and $a = \mu(x)$. Then, $a_x \in \mu$. By abnormality of $\eta$ in $\mu$, $a_x \in \langle \eta, \eta^{a_x} \rangle$. By assumption, $\eta^{a_x} \subseteq \theta$. Moreover, $\eta \subseteq \theta$. Hence $\langle \eta, \eta^{a_x} \rangle \subseteq \theta$. This implies $a_x \in \theta$, that is, $\theta(x) \geq a =\mu(x)$. Thus $\theta(x) = \mu(x)$ and we conclude that $\theta = \mu$. 
\end{proof}

\begin{example}
	Consider the $L$-subgroup $\eta$ of the $L$-group $\mu$ discussed in Example 1. We show that $\eta$ is a contranormal $L$-subgroup of $\mu$. Suppose that there exists an $L$-subgroup $\theta$ of $\mu$ such that $\eta^{a_x} \subseteq \theta$ for all $a_x \in \mu$. We show that $\theta = \mu$. For instance, consider $(123) \in S_4$. By assumption, 
	\[ a = u \wedge a = u \wedge \eta((123)) = \eta^{u_\epsilon}((123)) \leq \theta((123)). \]
	Also, 
	\[ b = d \wedge b = d \wedge \eta((124)) = \eta^{d_{(34)}}((123)) \leq \theta((123)). \] 
	Thus $\theta((123)) \geq d = a \vee b$. Since $\theta \in L(\mu)$, it follows that $\theta((123)) = \mu((123)) = d$. Similarly, it can be verified that for all $x \in G$, $\theta(x) = \mu(x)$, that is, $\theta = \mu$. It follows that $\eta$ is a contranormal $L$-subgroup of $\mu$.
	
\end{example}

\section*{Acknowledgements}
The second author of this paper was supported by the Senior Research Fellowship jointly funded by CSIR and UGC, India during the course of development of this paper.

\end{document}